\documentclass{amsart}

\newtheorem{theorem}{Theorem}[section]
\newtheorem{lemma}[theorem]{Lemma}
\newtheorem{example}[theorem]{Example}
\newtheorem{corollary}[theorem]{Corollary}

\theoremstyle{remark}
\newtheorem{remark}[theorem]{Remark}

\numberwithin{equation}{section}

\begin{document}

\title{Pointwise estimates for the derivative of algebraic polynomials}

\author{Adrian Savchuk}
\address{Taras Shevchenko National University of Kyiv, Kyiv, Ukraine}
\curraddr{Faculty of Mathematics and Mechanics, Academician Glushkov Avenue 4,
 02000, Kyiv, Ukraine}
\address{University of Toulon, Toulon, France}
\curraddr{UFR Sciences et Techniques, 83130, La Garde, France}
\email{adrian.savchuk.v@gmail.com, adrian-savchuk@etud.univ-tln.fr}


\subjclass[2020]{Primary 30A08,	30E10; Secondary 41A20}



\keywords{Algebraic polynomial,  Logarithmic derivative, Bernstein inequality}

\begin{abstract}
 We  give the sufficient condition on coefficients $a_k$ of an algebraic polynomial $P(z)=\sum_{k=0}^{n}a_kz^k$, $a_n\not=0,$ for the pointwise Bernstein inequality $|P'(z)|\le n|P(z)|$
to be true for all $z\in\overline{\mathbb D}:=\{w\in\mathbb C : |w|\le 1\}$.
\end{abstract}

\maketitle

\section{Introduction and main result}
Let $P$ be an algebraic polynomial with complex coefficients and let $z_1, z_2,\ldots, z_m$ be a distinct zeros of $P$ with multiplicities $r_1, r_2,\ldots, r_m$ respectively, $\sum_{k=1}^{m}r_k=\deg P$. Further we assume that $z_k$ are numerated in arbitrary manner so that $|z_1|\le|z_2|\le\cdots\le|z_m|$.

Consider the real part of the logarithmic derivative of $P$. We have
\begin{eqnarray}\label{fraction}
\mathop{\rm Re}\frac{zP'(z)}{P(z)}&=&\mathop{\rm Re}\sum_{k=1}^{m}\frac{r_kz}{z-z_k}\\
&=&\frac{n}{2}+\frac{1}{2}\sum_{k=1}^{m}r_k\frac{|z|^2-|z_k|^2}{|z-z_k|^2}\label{fraction2},
\end{eqnarray}
where $n=\deg P$. It follows that 	for all $z\in\mathbb C\setminus\{z_1,\ldots, z_m\}$,
\begin{equation}\label{lower est}
\left|\frac{n}{2}+\frac{1}{2}\sum_{k=1}^{m}r_k\frac{|z|^2-|z_k|^2}{|z-z_k|^2}\right|\le
\left|\frac{zP'(z)}{P(z)}\right|.
\end{equation}

Denote $\mathbb D:=\{z\in\mathbb D : |z|<1\},$ $\mathbb T:=\{z\in\mathbb C : |z|=1\}$ and we let by $d\sigma$ denote the normalized Lebesgue measure on $\mathbb T$. Assume $z_k\not\in\mathbb T$, $k=1,\ldots,m,$ then by integrating the last inequality along $\mathbb T$ we get
\begin{equation}\label{lower est int}
\sum_{k=1}^{j}r_k\le\int_{\mathbb T}\left|\frac{P'}{P}\right|d\sigma\le\max_{z\in\mathbb T}\left|\frac{P'(z)}{P(z)}\right|,
\end{equation}
where $j$ is positive integer $\le m$ such that $|z_j|<1<|z_{j+1}|$. Here and further we put $\sum_{k=1}^{0}=0$. 

From (\ref{lower est}), (\ref{lower est int}) and from the well-known Bernstein inequality, that say that 
\begin{equation}\label{B1}
\max_{z\in\mathbb T}|P'(z)|\le n\max_{z\in\mathbb T}|P(z)|,
\end{equation}
we readily conclude that for any algebraic polynomial $P$, $\deg P=n,$  having all its zeros in $\mathbb D$, the following inequalities holds
\begin{equation}\label{ineq}
\frac{n}{1+|z_m|}\le\min_{t\in\mathbb T}\left|\frac{P'(z)}{P(z)}\right|\le n\le \max_{z\in\mathbb T}\left|\frac{P'(z)}{P(z)}\right|.
\end{equation}
The first inequality was observed by Govil \cite{Gov}, the second one is the consequence of (\ref{B1}) and the third one is the consequence of (\ref{lower est int}). All these results are sharp. The equalities occurs for the polynomial $P(z)=a_n(z-c)^n$ for suitable $c\in\mathbb D$.

Assume now that all zeros $z_1,\ldots, z_m$ of $P$ lies in the domain $\mathbb U:=\{z\in\mathbb C : |z|\ge1\}$. Then it follows from (\ref{fraction2}) that
\[
\mathrm{Re}\frac{zP'(z)}{P(z)}\le\frac{n}{2}
\]
for all $z\in\overline{\mathbb D}\setminus\{z_1,\ldots,z_m\}$. This gives, as was noted by Aziz \cite{Aziz} (see also Lemma \ref{prop1} below),
\begin{equation}\label{Aziz ineq}
|zP'(z)|\le|nP(z)-zP'(z)|
\end{equation}
for all $z\in\overline{\mathbb D}$.

On the other side, it is easy to see that if $P$ is a  polynomial of degree $n$ having all its zeros in $\mathbb U_2:=\{z\in\mathbb C : |z|\ge 2\}$, then
\begin{eqnarray*}
	\max_{z\in\overline{\mathbb D}}\left|\frac{zP'(z)}{P(z)}\right|&\le&\sum_{k=1}^{m}\frac{r_k}{|z_k|-1}\\
	&\le&n.
\end{eqnarray*}
This is equivalent to
\begin{equation}\label{B2}
|zP'(z)|\le n|P(z)|
\end{equation}
for all $z\in\overline{\mathbb D}$. We will call the last relation a pointwise Bernstein inequality.

Combining (\ref{Aziz ineq}) and (\ref{B2}), we obtain, for all $z\in\overline{\mathbb D}$,
\begin{equation}\label{B4}
|zP'(z)|\le\min\left(|nP(z)-zP'(z)|, n|P(z)|\right),
\end{equation}
provided $\{z_1,\ldots,z_m\}\in\mathbb U_2$. For example, if $P(z)=(2+z)^n$, then (\ref{B4}) gives
\[
|z|\le\min\left(2,|2+z|\right)=
\begin{cases}
2,\hfill&\mbox{if}~z\in\overline{\mathbb D}\setminus\{w\in\mathbb C: |2+w|\le 2\},\cr
|2+z|,\hfill&\mbox{otherwise},
\end{cases}
\]
for all $z\in\overline{\mathbb D}$. Equality occurs here only in the point $z=-1$.

In this note we  give the sufficient condition on coefficients of a polynomial $P$ for the pointwise Bernstein inequality
to be true for all $z\in\overline{\mathbb D}$. As we will see, this condition implies (\ref{B4}) and does not require that all zeros  of $P$ must be in $\mathbb U_2$. 

For further information about the estimates of derivative  and the logarithmic derivative of polynomials we refer to \cite{Mil}, \cite{Rah}, \cite{Sheil-Small}, \cite{Danch}  and references therein. 

Our main result is the following theorem.

\begin{theorem}\label{main thm} Let $n\in\mathbb Z_+$ and $\{k_\nu\}_{\nu=0}^n$, $0\le k_0<k_1<\ldots<k_n$, be positive integers and let $P(z)=\sum_{\nu=0}^{n}a_\nu z^{k_\nu}$ be an algebraic polynomial of degree $k_n$ with coefficients $\{a_\nu\}_{\nu=0}^n\in\mathbb C\setminus\{0\}$. If
	\begin{equation}\label{main condition}
	\mathop{\rm \min_{t\in\overline{\mathbb D}}Re}\sum_{j=0}^{n-\nu}\frac{a_{j+\nu}}{a_{\nu}}t^{k_{j+\nu}-k_\nu}\ge\frac{1}{2},\quad\nu=0,1,\ldots,n,
	\end{equation}
	then the following assertions are holds true:
	
	(i) the polynomial $P$ have no zeros in $\overline{\mathbb D}$, provided  $k_0=0$, and have no zeros in $\overline{\mathbb D}\setminus\{0\}$ for $k_0>0$;
	
	(ii) for all $z\in\overline{\mathbb D}$
	\begin{equation}\label{B3}
	|zP'(z)|\le k_n|P(z)|.
	\end{equation}
	
	If $z\in\mathbb D$ the equality occurs here only in case $n=0$, that is for $P(z)=a_0z^{k_0}$, $k_0>0$;
	
	(iii) if $k_0=0$ and $n\ge 1$, then we have
	\begin{equation}\label{strong B}
	|P'(z)|< k_n|P(z)|
	\end{equation}
	for all $z\in\mathbb D$.
\end{theorem}

\begin{remark}
	Let $P$ be as in Theorem \ref{main thm}. Then we have the implication $(ii)\Rightarrow(i)$.
\end{remark}
This is a consequence of the Riemann's theorem on removable singularities applied to the function
\[
z\mapsto\frac{zP'(z)}{P(z)}=\sum_{k=1}^{m}\frac{r_kz}{z-z_k}.
\]

\begin{corollary}
	Let $P$ be as in Theorem \ref{main thm} with $a_0\ge a_1\ge\ldots\ge a_{n}>0$, $n\in\mathbb N$ and $k_0=0$. If 
	\[
	0\le\Delta^2 (a_\nu):=
	\begin{cases}
	a_{\nu+2}-2a_{\nu+1}+a_\nu,\hfill&\mbox{if}~\nu=0,1,\ldots,n-2,\cr
	a_{n-1}-2a_n,\hfill&\mbox{if}~\nu=n-1,\cr
	a_n,\hfill&\mbox{if}~\nu=n,
	\end{cases}
	\]
	then there holds
	\[
	|zP'(z)|\le\min\left(|k_nP(z)-zP'(z)|, k_n|P(z)|\right)
	\]
	for all $z\in\overline{\mathbb D}$.
\end{corollary}

Indeed, for each $\nu=0,1,\ldots, n$ the sequence $\{\lambda_{k,\nu}\}_{k=0}^{n-\nu+1},$ where
\[
\lambda_{k,\nu}=
\begin{cases}
\displaystyle\frac{a_{k+\nu}}{a_\nu},\hfill&\mbox{if}~k=0,1,\ldots,n-\nu,\cr
0,\hfill&\mbox{if}~k=n-\nu+1,
\end{cases}
\]
is non-negative, monotonically non-increasing and convex, i. e. $\lambda_{0,\nu}\ge\lambda_{1,\nu}\ge\ldots\ge\lambda_{n-\nu,\nu}>\lambda_{n-\nu+1,\nu}=0$ and $\Delta^2(\lambda_{k,\nu})\ge0$ for $k=0,1,\ldots,n-\nu+1$. Thus by Fej\'er theorem (see \cite{Mil}, p. 310) the trigonometric polynomials
\[
\frac{\lambda_{0,\nu}}{2}+\sum_{k=1}^{n-\nu}\lambda_{k,\nu}\cos kx,\quad\nu=0,1,\ldots,n,
\]
are non-negative for all $x\in\mathbb R$. This is equivalent to the condition (\ref{main condition}).

\begin{example}
	Let $n\in\mathbb N\setminus\{1\}$  and let
	\[
	P(z)=\sum_{k=0}^{n}(n+1-k)z^k.
	\]
	Then for $t=\mathrm e^{\mathrm ix},$ $x\in\mathbb R$, we have
	\begin{eqnarray*}
		\frac{1}{2}+\mathop{\rm Re}\sum_{k=1}^{n-\nu}\frac{n+1-(k+\nu)}{n+1-\nu}t^k&=&\frac{1}{2}+\sum_{k=1}^{n-\nu}\left(1-\frac{k}{n+1-\nu}\right)\cos kx\\
		&=&F_{n-\nu+1}(x)\ge 0,
	\end{eqnarray*}
	for all $x\in\mathbb R,~\nu=0,1,\ldots,n$, where $F_k$ is the Fej\'er kernel (see \cite{Mil}, pp. 311, 313).
	
	Therefore, combining (\ref{Aziz ineq}) and (\ref{B3}), we get
	\[
	\left|\sum_{k=1}^{n}(n+1-k)kz^k\right|\le\min\left(\left|\sum_{k=0}^{n-1}(n+1-k)(n-k)z^k\right|,n\left|\sum_{k=0}^{n}(n+1-k)z^k\right|\right). 
	\]
	
	By Enestr\"om--Kakeya theorem (see \cite{Rah}, p.255) with refinement given by Anderson, Saff and Varga \cite{And} (see Corollary 2), zeros of $P$ satisfy $|z_k|<2$, $k=1,\ldots,n$. 
\end{example}


\section{Lemmas}

For the proof of Theorem 1.1 we require the following lemmas.

\begin{lemma}\label{prop1}
	Let $P$ and $Q$ be a functions defined on a compact set $K\subset\mathbb C$, $\mathcal Z(Q):=\{z\in\mathbb C : Q(z)=0\}$ and $K\setminus\mathcal Z(Q)\not=\emptyset$. In order that 
	\[
	|P(z)-Q(z)|\le|P(z)|
	\]
	for all $z\in K$ it is necessary and sufficient that
	\[
	\inf_{z\in K\setminus\mathcal Z(Q)}\mathop{\rm Re}\frac{P(z)}{Q(z)}\ge\frac{1}{2}.
	\]
\end{lemma}

\begin{proof}The assertion readily follows from the obvious identity 
	\[
	\left|\frac{P(z)}{Q(z)}\right|^2-\left|\frac{P(z)}{Q(z)}-1\right|^2=2\mathbb{\rm Re}\frac{P(z)}{Q(z)}-1
	\]
	for $z\in K\setminus\mathcal Z(Q)$.
\end{proof}

\begin{lemma}
	\label{lemma2}
	Let $P(z)=\sum_{j=0}^{n}a_jz^j$, $n\in\mathbb N,$ and $a_n\not=0$. Then for all $z\in\mathbb C\setminus\{0\}$ and $w\in\mathbb C$ we have
	\[
	\left|z\frac{P(z)-P(w)}{z-w}\right|\le A(z,w)\max_{k=0,\ldots,n-1}\left|P(z)-\sum_{j=0}^{k}a_jz^j\right|,
	\]
	where
	\[
	A(z,w)=
	\begin{cases}
	\displaystyle \frac{|z|^n-|w|^n}{|z|^{n-1}(|z|-|w|)},\hfill&\mbox{if}~|z|\not=|w|,\cr
	n,\hfill&\mbox{if}~|z|=|w|.
	\end{cases}
	\]
	The result is best possible and the equality holds for the polynomial $P(z)=a_0+a_nz^n$ in case $\arg z=\arg w$.
\end{lemma}

\begin{proof} Fix $z\in\mathbb C\setminus\{0\}$. Summation by parts yields
	\[
	P(w)=
	P(z)\left(\frac{w}{z}\right)^n+\left(1-\frac{w}{z}\right)\sum_{k=1}^{n}\left(\sum_{j=0}^{k-1}a_jz^j\right)\left(\frac{w}{z}\right)^{k-1}.
	\]
	This gives
	\begin{equation}\label{main formula}
	z\frac{P(z)-P(w)}{z-w}=\sum_{k=0}^{n-1}\left(P(z)-\sum_{j=0}^{k}a_jz^j\right)\left(\frac{w}{z}\right)^k.
	\end{equation}
	The result follows.
\end{proof}

\section{Proor of Theorem \ref{main thm}}

Denote
\[
\rho_k(P)(z):=\sum_{j=k}^{k_n}c_jz^j,~k=0,1,\ldots,k_n,
\]
where
\[
c_j=
\begin{cases}
0,\hfill&\mbox{if}~j\not\in\{k_\nu\}_{\nu=0}^n,\cr
a_j,\hfill&\mbox{if}~j\in\{k_\nu\}_{\nu=0}^n.
\end{cases}
\]

(i) By Lemma \ref{prop1} the conditions (\ref{main condition}) are equivalent to
\begin{equation}\label{decrease ineq}
|P(z)|\ge|\rho_{k_0}(P)(z)|\ge\ldots\ge|\rho_{k_n}(P)(z)|=\left|a_{n}z^{k_n}\right|\quad\forall z\in\overline{\mathbb D}.
\end{equation}
This gives that $P(z)\not=0$ for $z\in\overline{\mathbb D}\setminus\{0\}$. If $k_0=0$ then in addition  $P(0)=a_{k_0}\not=0$.

(ii) It follows from (\ref{decrease ineq}) that the sequence $\{|\rho_{k_\nu}(P)(z)|\}_{\nu=0}^n$ is non-increasing. Since $\rho_{j}(P)=\rho_{k_\nu}(P)$
for $k_{\nu-1}<j\le k_{\nu}$, $\nu=0,1,\ldots,n,$ where $k_{-1}=-1$, we conclude that the sequence $\{|\rho_j(P)(z)|\}_{j=0}^{k_n}$ also is non-increasing. Therefore by Lemma \ref{lemma2} we get
\begin{eqnarray*}
	\left|z\frac{P(z)-P(zt)}{1-t}\right|&\le& k_n|\rho_{n_0}(P)(z)|\\
	&\le&k_n|P(z)|
\end{eqnarray*}
for all $t\in\mathbb T$. In particularly, for $t=1$ we obtain (\ref{B3}). 

Now assume that for some $z\in\mathbb D$ in (\ref{B3}) occurs equality. Then according to the assertion (i) proved above, the function
\[
F(t):=\frac{tP'(t)}{k_nP(t)}=\frac{k_0}{k_n}+\frac{(k_1-k_0)a_1}{k_na_0}t^{k_1-k_0}+\ldots
\]
is holomorphic in $\mathbb D$, $|F(t)|\le 1$ for all $t\in\mathbb D$ and $|F(z)|=1$. Therefore by maximum modulus principle $F(t)=c$ for all $t\in\mathbb D$ with $|c|=1$. But $F(0)=k_0/k_n$. So $c=k_0/k_n=1$. These are equivalent to $n=0$ and $P(t)=\mathrm{e^M}t^{k_0}$ for some $M\in\mathbb C$.

(iii) Let $k_0=0$. Then according to the above conclusions about the function $F$, we have that $F(0)=0$. Therefore  by Schwarz lemma we get $|F(t)|\le|t|$ for all $t\in\mathbb D$. Moreover, if $|F(z)|=|z|$ for some $z\in\mathbb D\setminus\{0\}$, then $F(t)=ct$ for some $c\in\mathbb C$ with $|c|=1$. Once this is done, it follows that  
\[
c=F'(t)=\frac{k_1^2a_1}{k_na_0}t^{k_1-1}+\cdots,\quad t\in\mathbb D.
\]
Hence, necessarily $k_1=1$ and $|a_1|=k_n|a_0|$. However, under condition (\ref{main condition}),
\begin{eqnarray*}
	\left|\frac{a_1}{a_0}\right|&=&\left|\int_\mathbb Tt^{k_1-k_0}\mathrm{Re}\left(1+2\sum_{j=1}^n\frac{a_j}{a_0}t^{k_j-k_0}\right)d\sigma(t)\right|\\
	&\le&\int_\mathbb T\mathrm{Re}\left(1+2\sum_{j=1}^n\frac{a_j}{a_0}t^{k_j-k_0}\right)d\sigma(t)\\
	&=&1.
\end{eqnarray*}
Thus $k_n=1$ or equivalently, $n=1$. On the other side, for $n=1$ the condition (\ref{main condition}) implies $|a_0|\ge 2|a_1|$. This is a contradiction. Hence, $|F(t)|<|t|$ for all $t\in\mathbb D$. 

The proof is complete.

\bibliographystyle{amsplain}

\begin{thebibliography}{10}

\bibitem{Gov}
Govil, N. K. {\it On the derivative of a polynomial.}
Proc. Amer. Math. Soc. \textbf{41} (1973), 543--546.

\bibitem{Aziz}
Aziz, A. {\it Integral mean estimates for polynomials with restricted zeros}. J. Approx. Theory \textbf{55} (1988), no. 2, 232--239.

\bibitem{Mil}
Milovanovi\'c, G. V., Mitrinovi\'c, D. S., Rassias, Th. M. 
{\it Topics in polynomials: extremal problems, inequalities, zeros}. World Scientific Publishing Co., Inc., River Edge, NJ, 1994, 821 pp.

\bibitem{Rah}
Rahman, Q. I., Schmeisser, G. {\it Analytic theory of polynomials.} London Mathematical Society Monographs. New Series, 26. The Clarendon Press, Oxford University Press, Oxford, 2002, 742 pp.

\bibitem{Sheil-Small}
Sheil-Small, T. {\it Complex polynomials.}
Cambridge Studies in Advanced Mathematics, 75. Cambridge University Press, Cambridge, 2002, 428 pp.

\bibitem{Danch}
Danchenko, V. I., Komarov, M. A., Chunaev, P. V. {\it Extremal and Approximative Properties of Simple Partial Fractions}. Russ Math. \textbf{62} (2018), 6--41. https://doi.org/10.3103/S1066369X18120022

\bibitem{And}
Anderson, N., Saff, E. B., 
Varga, R. S. {\it On the Enestr\"om-Kakeya theorem and its sharpness}.
Linear Algebra Appl. \textbf{28} (1979), 5--16.

\end{thebibliography}

\end{document}